\documentclass[11pt,reqno]{amsart}

\usepackage{amsmath}
\usepackage{amsfonts}

\newcommand{\trinorm}{|\!|\!|}
\newcommand{\rr}{{\mathbb R}}

\newcommand{\N}{{\mathbb N}}

\newcommand{\ff}{\varphi}

\newtheorem{theorem}{Theorem}[section]
\newtheorem{lemma}[theorem]{Lemma}
\newtheorem{corollary}[theorem]{Corollary}
\newtheorem{proposition}[theorem]{Proposition}

\theoremstyle{definition}

\theoremstyle{remark}

\numberwithin{equation}{section}

\begin{document}

\title[gZK] {Local and global well-posedness for the 2D generalized Zakharov-Kuznetsov equation}

\author[F. Linares  And A. Pastor]{Felipe Linares and Ademir Pastor}

\email{linares@impa.br, apastor@impa.br}




\maketitle


\begin{center}
{Instituto Nacional de Matem\'atica Pura e Aplicada - IMPA,\\
Estrada Dona Castorina 110, CEP 22460-320, Rio de Janeiro, RJ,
Brazil.     }
\end{center}

\begin{abstract}
This paper addresses well-posedness issues for the initial value
problem (IVP) associated with the generalized Zakharov-Kuznetsov
equation, namely,
\begin{equation*}
\quad  \left\{
\begin{array}{lll}
{\displaystyle u_t+\partial_x \Delta u+u^ku_x  =  0,}\qquad (x,y) \in \mathbb{R}^2, \,\,\,\, t>0,\\
{\displaystyle  u(x,y,0)=u_0(x,y)}.
\end{array}
\right.
\end{equation*}
For $2\leq k \leq 7$, the IVP above is shown to be locally
well-posed for data in $H^s(\mathbb{R}^2)$, $s>3/4$. For $k\geq8$,
local well-posedness is shown to hold for data in
$H^s(\mathbb{R}^2)$, $s>s_k$, where $s_k=1-3/(2k-4)$. Furthermore,
for $k\geq3$, if $u_0\in H^1(\mathbb{R}^2)$ and satisfies
$\|u_0\|_{H^1}\ll1$, then the solution is shown to be global in
$H^1(\mathbb{R}^2)$. For $k=2$, if $u_0\in H^s(\mathbb{R}^2)$,
$s>53/63$, and satisfies $\|u_0\|_{L^2}<\sqrt3 \,
\|\varphi\|_{L^2}$, where $\varphi$ is the corresponding ground
state solution, then the solution is shown to be global in
$H^s(\mathbb{R}^2)$.
\end{abstract}


\section{Introduction}

This paper is concerned with the initial value problem (IVP)
associated with the generalized Zakharov-Kuznetsov (gZK) equation in
two space dimensions, namely,
\begin{equation}\label{IVP}
\left\{
\begin{array}{lll}
{\displaystyle u_t+\partial_x \Delta u+u^ku_x  =  0,  }  \qquad (x,y) \in \mathbb{R}^2, \,\,\,\, t>0, \\
{\displaystyle  u(x,y,0)=u_0(x,y)},
\end{array}
\right.
\end{equation}
where $u$ is a real-valued function, and $k\geq2$ is an integer
number.

When $k=1$, the equation in \eqref{IVP} (termed simply as ZK
equation) was formally deduced by Zakharov and Kuznetsov \cite{ZK}
(see also \cite{LS} and references therein) as an asymptotic model
to describe the propagation of nonlinear ion-acoustic waves in a
magnetized plasma. The equation in \eqref{IVP} may also be seen as a
natural, two-dimensional extension of the one-dimensional
generalized Korteweg-de Vries (KdV) equation
$$
u_t+u_{xxx}+u^ku_x  =  0.
$$

The aim of this paper is to establish local and global
well-posedness to the IVP \eqref{IVP}. The notion of well-posedness
will be the usual one in the context of nonlinear dispersive
equations, that is, it includes existence, uniqueness, persistence
property, and continuous dependence upon the data.

Before describing our results, let us recall what has been done so
far regarding the gZK equation. In \cite{Fa}, Faminskii considered
the IVP associated with the ZK equation. He showed local and global
well-posedness for initial data in $H^m(\mathbb{R}^2)$, $m\geq1$ integer.
In \cite{BL}, Biagioni and Linares dealt with the IVP
associated with the modified ZK equation (i.e. that one in
\eqref{IVP} with $k=2$). They proved local and global well-posedness
for data in $H^1(\mathbb{R}^2)$. Linares and Pastor (\cite{LP})
studied the IVP associated with both the ZK and modified ZK
equations. They improved the results in \cite{BL}, \cite{Fa} by
showing that both IVP's are locally well-posed for initial data in
$H^s(\mathbb{R}^2)$, $s>3/4$. Moreover, by using the techniques
introduced in Birnir \textit{at al.} \cite{BPS}, \cite{BKPSV}, they
proved that the IVP associated with the modified ZK equation is
ill-posed, in the sense that the flow-map data-solution is not
uniformly continuous, for data in $H^s(\mathbb{R}^2)$, $s\leq0$. It
should be noted that the method employed in \cite{BL}, \cite{Fa},
\cite{LP} to show local well-posedness, was the one developed by
Kenig, Ponce, and Vega \cite{KPV} (when dealing with the generalized
KdV equation), which combines smoothing effects, Strichartz-type
estimates, and a maximal function estimate together with the Banach
contraction principle.

It is worth mentioning that in \cite{LP}, the authors proved that if
$u_0\in H^1(\mathbb{R}^2)$ and satisfies $\|u_0\|_{L^2}<\sqrt3 \,
\|\varphi\|_{L^2}$, where $\varphi$ is the unique positive radial
solution (hereafter refereed to as the ground state solution) of the
elliptic equation
\begin{equation}\label{solwave}
-\Delta\varphi+\varphi-\varphi^3=0,
\end{equation}
then the solution $u(t)$ of \eqref{IVP}, with $k=2$, may be globally
extended in $H^1(\mathbb{R}^2)$. It should be pointed out that if
$\|u_0\|_{L^2}\geq\sqrt3 \, \|\varphi\|_{L^2}$, the question of
showing whether or not the $H^1$-solution, with $u(0)=u_0$, blows up
in finite time is currently open.


It should also be observed that questions of existence and orbital
stability of solitary-wave solutions, and unique continuation  were
addressed, respectively, by de Bouard \cite{dB}, and Panthee
\cite{Pa}. In \cite{dB}, the author proved that the positive,
radially symmetric  solitary waves are orbitally stable if $k=1$,
and orbitally unstable otherwise. In \cite{Pa}, the author
established that if the solution of the ZK equation is sufficiently
regular and is compactly supported in a nontrivial time interval,
then it vanishes identically.

Now, let us describe our results. We first recall that the
quantities
\begin{equation}\label{mass}
I_1(u(t))=\int_{\mathbb{R}^2}u^2(t)\;dxdy
\end{equation}
and
\begin{equation}\label{energy}
I_2(u(t))=\int_{\mathbb{R}^2}\left\{u_x^2(t)+u_y^2(t)-\frac{2}{(k+1)(k+2)}u^{k+2}(t)\right\}\;dxdy
\end{equation}
are conserved by the flow of the gZK equation, that is,
$I_1(u(t))=I_1(u(0))$ and $I_2(u(t))=I_2(u(0))$, as long as the
solution exists. Thus, these quantities could lead local rough
solutions to global ones. So, it is natural to ask what would be the
largest Sobolev space where local well-posedness holds. To answer
this question, we perform a scaling argument, by noting that if $u$
solves \eqref{IVP}, with initial data $u_0$, then
$$
u_\lambda(x,y,t)=\lambda^{2/k} u(\lambda x, \lambda y, \lambda^3t)
$$
also solves \eqref{IVP}, with initial data
$u_\lambda(x,y,0)=\lambda^{2/k} u_0(\lambda x, \lambda y)$, for any
$\lambda>0$. Hence,
\begin{equation}\label{scaling}
\|u(\cdot,\cdot,0) \|_{\dot{H}^s}=\lambda^{2/k+s-1}
\|u_0\|_{\dot{H}^s}
\end{equation}
where $\dot{H}^s=\dot{H}^s(\mathbb{R}^2)$ denotes the homogeneous
Sobolev space of order $s$. As a consequence of \eqref{scaling}, the
scale-invariant Sobolev space for the gZK equation is
$H^{s_c(k)}(\mathbb{R}^2)$, where $s_c(k)=1-2/k$. Therefore, one expects
that the Sobolev spaces $H^{s}(\mathbb{R}^2)$ for studying the well-posedness of
\eqref{IVP} are those with indices $s>s_c(k)$.

We divide the paper into two parts. The first one deals with local
and global well-posedness in the case $k\geq3$, whereas the second
part is devoted to establishing the global well-posedness in the
case $k=2$ (the critical case).

Our first result regards local well-posedness of \eqref{IVP} for $3\leq k\leq7$.
More precisely, we prove the following.

\begin{theorem}   \label{theorem1}
Assume $3\leq k\leq 7$. For any $u_0 \in H^s(\mathbb{R}^2)$,
$s>3/4$, there exist $T=T(\|u_0\|_{H^s})>0$ and a unique solution of
the IVP \eqref{IVP}, defined in the interval $[0,T]$, such that
\begin{equation}\label{b1}
u \in C([0,T];H^s(\mathbb{R}^2)),
\end{equation}
\begin{equation}\label{b2}
\|D^s_xu_x\|_{L^\infty_xL^2_{yT}}  +
\|D^s_yu_x\|_{L^\infty_xL^2_{yT}}  <\infty,
\end{equation}
\begin{equation}\label{b3}
 \|u\|_{L^{p_k}_T L^\infty_{xy}}+   \|u_x\|_{L^{12/5}_T L^\infty_{xy}}
  <\infty,
\end{equation}
and
\begin{equation}\label{b4}
\|u\|_{L^4_x L^\infty_{yT}}
  <\infty,
\end{equation}
where $p_k=\frac{12(k-1)}{7-12\gamma}$ and $\gamma\in(0,1/12)$.
Moreover, for any $T'\in(0,T)$ there exists a neighborhood $W$ of
$u_0$ in $H^s(\mathbb{R}^2)$ such that the map
$\widetilde{u}_0\mapsto \widetilde{u}(t)$ from $W$ into the class
defined by \eqref{b1}--\eqref{b4} is smooth.
\end{theorem}

To prove Theorem \ref{theorem1} we use the technique
introduced by Kenig, Ponce, and Vega \cite{KPV} to study the IVP associated to the KdV equation.
We point out that the proof of Theorem 1.1 in \cite{LP} does not apply to the case
$k\geq3$. Here, instead of using an $L^2_x$ maximal function
estimate, we use an $L^4_x$ one (see Proposition \ref{proposition2}
below). However, the main new ingredient is the embedding given in
Lemma \ref{lemma4}. Observe that for $3\leq k \leq 7$, we obtain
$s_c(k)<3/4$, so that, our result does not reach the indices conjectured
by the scaling argument.

Next, we deal with the case $k\geq8$. Our main result in this case
reads as follows.

\begin{theorem}   \label{theorem2}
Let $k\geq8$ and $s_k=1-\frac{3}{2(k-2)}$. For any $u_0 \in
H^s(\mathbb{R}^2)$, $s>s_k$, there exist $T=T(\|u_0\|_{H^s})>0$ and
a unique solution of the IVP \eqref{IVP}, defined in the interval
$[0,T]$, such that
\begin{equation}\label{b1.1}
u \in C([0,T];H^s(\mathbb{R}^2)),
\end{equation}
\begin{equation}\label{b2.1}
\|D^s_xu_x\|_{L^\infty_xL^2_{yT}}  +
\|D^s_yu_x\|_{L^\infty_xL^2_{yT}}  <\infty,
\end{equation}
\begin{equation}\label{b3.1}
 \|u\|_{L^{\widetilde{p}_k}_T L^\infty_{xy}}+   \|u_x\|_{L^{12/5}_T L^\infty_{xy}}
  <\infty,
\end{equation}
and
\begin{equation}\label{b4.1}
\|u\|_{L^4_x L^\infty_{yT}}
  <\infty,
\end{equation}
where $\widetilde{p}_k=\frac{2(k-2)}{1-2\gamma}$ and $\gamma>0$ is
sufficiently small. Moreover, for any $T'\in(0,T)$ there exists a
neighborhood $U$ of $u_0$ in $H^s(\mathbb{R}^2)$ such that the map
$\widetilde{u}_0\mapsto \widetilde{u}(t)$ from $U$ into the class
defined by \eqref{b1.1}--\eqref{b4.1} is smooth.
\end{theorem}

The proof of Theorem \ref{theorem2} is very close to that of Theorem
\ref{theorem1}. In this case, because of the scaling argument, we do
not expect to prove local well-posedness for all $s>3/4$. Indeed,
note that, for $k\geq8$, we always have, $s_k\geq s_c(k)\geq3/4$.
Moreover, $s_k= s_c(k)$ if and only if $k=8$. Also observe that in
the case $k=8$, we get $s_8=s_c(8)=3/4$. This implies that our
result, for $k=8$, is ``almost'' sharp, but for $k>8$ there is still
a gap between the scaling and our results, which is evidenced by the
theorem below.

\begin{theorem}\label{ill-posedness}
Let $k\geq3$. Then, the IVP \eqref{IVP} is ill-posed for data in
$H^{s_c(k)}(\mathbb{R}^2)$, $s_c(k)=1-2/k$, in the sense that the
map data-solution is not uniformly continuous.
\end{theorem}

Note that the  well-posedness sense in Theorem \ref{ill-posedness}
requires additional smoothness of the map data-solution, and not
only that of merely continuity. However, this is not too strong
because as affirmed in Theorems \ref{theorem1} and \ref{theorem2},
the map data-solution, in those cases, is sufficiently smooth. The
argument to establish Theorem \ref{ill-posedness} is similar to that
of Theorem 1.2 in \cite{LP}, and goes back to the techniques
introduced in \cite{BPS} and \cite{BKPSV}.

One of the main difficulties to obtain possible sharp results is the lack
of some needed estimates in mixed spaces. There is not an available Leibniz rule for fractional
derivatives in mixed spaces $L^q_xL^p_yL^r_T$ for instance. This makes a
difference with the analysis for the generalized KdV for $k\ge 4$. Another
point we should remark is the gain of derivatives we have for the Strichartz estimate
for the linear group. We only get $1/4-\epsilon$ derivatives, $0<\epsilon\ll 1$,
(see Lemma \ref{proposition1} below)
in contrast to the gain of exactly 1/4 derivatives of the KdV linear group.
Because of this we also loose some regularity.

We now turn our attention to the global well-posedness issue.
Our main result is proved under a smallness condition on the initial
data.

\begin{theorem}    \label{globaltheorem}
Let $k\geq3$. Let $u_0\in H^1(\mathbb{R}^2)$ and assume
$\|u_0\|_{H^1}\ll1$, then the local solutions given in Theorems
\ref{theorem1} and \ref{theorem2} can be extended to any time
interval $[0,T]$.
\end{theorem}

Theorem \ref{globaltheorem} is proved in a standard fashion, and
relies on a combination of the conserved quantities \eqref{mass} and
\eqref{energy} with the Gagliardo-Nirenberg interpolation inequality.

Next, we will focus on the second part of the paper.
As we already mentioned, the local well-posedness of \eqref{IVP} with $k=2$
for initial data in $H^s(\mathbb{R}^2)$, $s>3/4$, was
obtained in \cite{LP}. Furthermore, we announced that
\eqref{IVP} was globally well-posed for the initial data $u_0$
in $H^s(\mathbb{R}^2)$, $s>19/21$ satisfying
$\|u_0\|_{L^2}<\sqrt3\|\varphi\|_{L^2}$, where $\varphi$ is the
ground state solution of equation \eqref{solwave}. In the present
paper, we reaffirm that this result holds, however,  we slightly
modify the proof of the local well-posedness in \cite{LP}, to
improve that announced result. More precisely, we prove the
following.

\begin{theorem}  \label{globalm}
Let $k=2$. Let $u_0\in H^s(\mathbb{R}^2)$, $s>53/63$, and assume
that $\|u_0\|_{L^2}<\sqrt3\|\varphi\|_{L^2}$,  where $\varphi$ is
the ground state solution of equation \eqref{solwave}, then
\eqref{IVP} is globally well-posed.
\end{theorem}

The method we use to prove Theorem \ref{globalm} is that one
developed in \cite{FLP} and \cite{FLP1}, which combines the
smoothing effects for the solution of the linear problem with the
iteration process introduced by Bourgain \cite{Bo1}. Since we are in
the critical case, as in \cite{FLP1}, controlling the $L^2$-norm of
the initial data could bring some difficult. Nevertheless, with a
suitable decomposition of the initial data into low and high
frequencies, we are able to handle this.

Let us highlight what enables us to improve the global result
announced in \cite{LP}. The reason is quite simple. In \cite{LP}, to
apply the contraction principle, we get a factor of $T^{2/3}$ in
front of the estimates for the nonlinear terms. Here, modifying a
little bit the functional spaces, we get a factor of $T^{5/12-}$ (see
proof of Theorem \ref{theorem3}), this in turn, is relevant in the
method described in \cite{FLP}, \cite{FLP1}.

As we have pointed out in \cite{LP}, the Fourier restriction method
does not seem to work to proving a local well-posedness result for
the generalized ZK equation. So, it is not clear that the I-method,
introduced by Colliander \textit{et al.} \cite{CKSTT},  work either
to establish a better global well-posedness result.

The paper is organized as follows. In Section \ref{linear}, we state
the results concerned with the linear problem associated with
\eqref{IVP}. In Section \ref{local}, we deal with the case $k\geq3$.
We show our local (and global) well-posedness result as well as the
ill-posedness one. Finally, in Section \ref{global} we establish the
global well-posedness for $k=2$ announced in  Theorem \ref{globalm}.

 \vskip.3cm

\noindent{\bf Notation.} The symbol $a\pm$ means that there exists an
$\varepsilon>0$, small enough, such that $a\pm=a\pm\varepsilon$. For
$\alpha \in \mathbb{C}$, the operators $D^\alpha_x$ and $D^\alpha_y$
are defined via Fourier transform by $\widehat{D^\alpha_x f}(\xi,
\eta)=|\xi|^\alpha\widehat{f}(\xi,\eta)$ and $\widehat{D^\alpha_y
f}(\xi, \eta)=|\eta|^\alpha\widehat{f}(\xi,\eta)$. The mixed
space-time norm is defined by (for $1\leq p,q,r<\infty$)
$$
\|f\|_{L^p_xL^q_yL^r_T}= \left( \int_{-\infty}^{+\infty} \left(
\int_{-\infty}^{+\infty} \left( \int_0^T |f(x,y,t)|^r dt
\right)^{q/r} dy \right)^{p/q} dx \right)^{1/p}.
$$
with obvious modifications if either $p=\infty$, $q=\infty$ or
$r=\infty$.

\section{Preliminary results} \label{linear}

In this section, we recall some results concerning
the linear IVP associated to the gZK equation, which will be useful
throughout the paper.

Consider the linear IVP
\begin{equation}\label{a1}
     \left\{
\begin{array}{lll}
{\displaystyle u_t+\partial_x \Delta u =  0,}  \qquad (x,t) \in \mathbb{R}^2, \,\,\,\, t \in \mathbb{R}. \\
{\displaystyle  u(x,y,0)=u_0(x,y)},
\end{array}
\right.
\end{equation}
The solution of \eqref{a1} is given by the unitary group
$\{U(t)\}_{t=-\infty}^\infty$ such that
\begin{equation}\label{a2}
u(t)=U(t)u_0(x,y)= \int_{\mathbb{R}^2}
e^{i(t(\xi^3+\xi\eta^2)+x\xi+y\eta)}\widehat{u}_0(\xi,\eta) d\xi
d\eta.\\
\end{equation}

We begin by remembering the smoothing effect of Kato type, and the
Strichartz-type estimates.

\begin{lemma}{\bf(Smoothing effect)}  \label{lemma1}
 Let $u_0 \in L^2(\mathbb{R}^2)$. Then,
  \begin{equation}  \label{seffect}
    \|\partial_x U(t)u_0\|_{L^\infty_xL^2_{yT}} \leq c
    \|u_0\|_{L^2_{xy}}
  \end{equation}
  and
\begin{equation}  \label{dseffect}
    \|\partial_x \int_0^tU(-t')f(\cdot,\cdot,t')dt'\|_{L^2_{xy}} \leq c
    \|f\|_{L^1_xL^2_{yT}}.
  \end{equation}
  Moreover, the same still hold if we replace $\partial_x$ with
  $\partial_y$.
\end{lemma}
\begin{proof}
See Faminskii \cite[Theorem 2.2]{Fa} for the proof of
\eqref{seffect}. The inequality \eqref{dseffect} is just the dual
version of \eqref{seffect}.\\
\end{proof}

\begin{proposition}{\bf({Strichartz-type estimates})}  \label{proposition1}
Let $0\leq \varepsilon <1/2$ and $0\leq \theta \leq 1$. Then, the
group $\{U(t)\}_{t=-\infty}^\infty$ satisfies
\begin{equation}\label{a3}
    \|D^{\theta \varepsilon/2}_x U(t)f \|_{L^q_tL^p_{xy}} \leq c
    \|f\|_{L^2_{xy}},
\end{equation}
where $p=\frac{2}{1-\theta}$ and
$\frac{2}{q}=\frac{\theta(2+\varepsilon)}{3}$.
\end{proposition}
\begin{proof}
See Linares and Pastor \cite[Proposition 2.4]{LP}.\\
\end{proof}

The next lemmas are useful to recover the ``loss of derivative''
present in the nonlinear term of the gZK equation.

\begin{lemma}   \label{lemma2}
Let $0\leq \varepsilon <1/2$. Then, the group
$\{U(t)\}_{t=-\infty}^\infty$ satisfies
\begin{equation}   \label{a4}
\|U(t)f \|_{L^p_TL^\infty_{xy}} \leq c T^{\gamma_1}
\|D^{-\varepsilon/2}_x f \|_{L^2_{xy}},
\end{equation}
where $1\leq p \leq \frac{6}{2+\varepsilon}$ and
$\gamma_1=\frac{1}{p}-\frac{2+\varepsilon}{6}$. In particular, if
$0<T\leq1$, then
\begin{equation}   \label{a5}
\|U(t)f \|_{L^{12/5}_TL^\infty_{xy}} \leq c \|D^{-\varepsilon/2}_x f
\|_{L^2_{xy}}.
\end{equation}
\end{lemma}
\begin{proof}
By using Holder's inequality (in $t$), we get
$$
\|U(t)f \|_{L^p_TL^\infty_{xy}} \leq c T^{\gamma_1}\|U(t)f
\|_{L^q_TL^\infty_{xy}},
$$
where $\frac{1}{p}=\gamma_1+\frac{1}{q}$. Thus, taking $\theta=1$
and $q=6/(2+\varepsilon)$ in Proposition \ref{proposition1},
the estimate \eqref{a4} then follows.\\
\end{proof}

\begin{lemma}   \label{lemma4}
Let $\delta>0$. Then,
$$
\|f\|_{L^\infty_{xy}}\leq c \left\{ \|f\|_{L^{p_\delta}_{xy}} +
\|D^\delta_xf\|_{L^{p_\delta}_{xy}}+\|D^\delta_yf\|_{L^{p_\delta}_{xy}}\right\},
$$
where $p_\delta>2/\delta$. In particular,
$p_\delta\rightarrow\infty$ as $\delta\rightarrow0$.
\end{lemma}
\begin{proof}
See Kenig and Ziesler \cite[Lemma 3.4]{KZ1}.\\
\end{proof}

\begin{lemma}   \label{lemma5}
Let $0<\delta<1$. Assume $1-\delta<\theta<1$ and $2\leq r \leq
3/\theta$. Then,
$$
\|U(t)f\|_{L^r_TL^\infty_{xy}}\leq c T^{\gamma_2} \left\{
\|f\|_{L^{2}_{xy}} +
\|D^\delta_xf\|_{L^{2}_{xy}}+\|D^\delta_yf\|_{L^{2}_{xy}}\right\},
$$
for some $\gamma_2=\frac{1}{r}-\frac{\theta}{3}\geq0$.
\end{lemma}
\begin{proof}
We first note that taking $\varepsilon=0$ in Proposition
\ref{proposition1}, we obtain
\begin{equation} \label{a6}
\|U(t)f\|_{L^{3/\theta}_TL^{2/(1-\theta)}_{xy}}\leq c
\|f\|_{L^2_{xy}}.
\end{equation}
Now, applying  Holder's inequality followed by Lemma \ref{lemma4},
we deduce
\begin{equation*}
\begin{split}
\|U(t)f\|_{L^r_TL^\infty_{xy}}&\;\leq c T^{\gamma_2}
\|U(t)f\|_{L^{r'}_TL^\infty_{xy}} \\
& \;\leq c T^{\gamma_2} \left\{
\|U(t)f\|_{L^{r'}_TL^{p_\delta}_{xy}}
+\|D^\delta_xU(t)f\|_{L^{r'}_TL^{p_\delta}_{xy}}
+\|D^\delta_yU(t)f\|_{L^{r'}_TL^{p_\delta}_{xy}} \right\}.
\end{split}
\end{equation*}
If we now choose $r'=3/\theta$ and $p_\delta=2/(1-\theta)$, then an
application of \eqref{a6} yields the affirmation. Note that
$p_\delta>2/\delta$ implies $1-\delta<\theta$, and $\gamma_2\geq0$
implies $r\leq3/\theta$. This completes the proof of the lemma.\\
\end{proof}

As we commented before, Kenig, Ponce, and Vega's technique combines
the smoothing effect and Strichartz estimate with a maximal function
estimate. Here, we present the $L^2_x$ and $L^4_x$ maximal function estimates
we will use in our arguments.

\begin{proposition}{\bf(Maximal function)}   \label{proposition2}

\begin{itemize}
  \item[(i)] For any $s_1>1/4$, $r_1>1/2$ and $0< T \leq 1$, we have
\begin{equation*}
\|U(t)f\|_{L^4_xL^\infty_{yT}} \leq c \|(1+D_x)^{s_1}(1+D_y)^{r_1}f
\|_{L^2_{xy}}.
\end{equation*}
  \item[(ii)] For any $s>3/4$, we have
  \begin{equation*}
\|U(t)f\|_{L^2_xL^\infty_{yT}} \leq c(s,T) \|f \|_{H^s_{xy}}.
\end{equation*}
where $c(s,T)$ is a positive constant depending only on $T$ and $s$.
\end{itemize}
\end{proposition}
\begin{proof}
See Linares and Pastor \cite[Proposition 1.5]{LP} for part (i), and Faminskii \cite[Theorem 2.4]{Fa} for part (ii).\\
\end{proof}

\begin{corollary}   \label{corollary1}
For any $s>3/4$ and $0< T \leq 1$, we have
\begin{equation*}
\|U(t)f\|_{L^4_xL^\infty_{yT}} \leq c \|f\|_{H^s_{xy}}.
\end{equation*}
\end{corollary}
\begin{proof}
Let $s_1$ and $r_1$ be as in Proposition \ref{proposition2}(i). In
view of Plancherel's theorem, we obtain
\begin{equation*}
\begin{split}
\|(1+D_x)^{s_1}(1+D_y)^{r_1}f
\|_{L^2_{xy}}^2&=\int_{\mathbb{R}^2}(1+|\xi|^2)^{s_1}(1+|\eta|^2)^{r_1}
|\widehat{f}(\xi,\eta)|^2d\xi d\eta \\
& \leq c
\int_{\mathbb{R}^2}(1+|\xi|^{2s_1}+|\eta|^{2r_1}+|\xi|^{2s_1}|\eta|^{2r_1})
|\widehat{f}(\xi,\eta)|^2d\xi d\eta \\
& \leq c
\int_{\mathbb{R}^2}(1+|\xi|^{2s_1}+|\eta|^{2r_1}+|\xi|^{6s_1}+|\eta|^{3r_1})
|\widehat{f}(\xi,\eta)|^2d\xi d\eta,
\end{split}
\end{equation*}
where in the last inequality we applied the Young inequality. Now,
splitting the integral into $B_1(0)$ and $\mathbb{R}^2\setminus
B_1(0)$, where $B_1(0)$ denotes the ball of radius 1 centered at the
origin, it is easy to see that
\begin{equation} \label{a7}
\begin{split}
\int_{\mathbb{R}^2}(1+|\xi|^{2s_1}+|\eta|^{2r_1}+|\xi|^{6s_1}+|\eta|^{3r_1})
|\widehat{f}(\xi,\eta)|^2d\xi d\eta \leq
c\int_{\mathbb{R}^2}(1+|\xi|^{6s_1}+|\eta|^{3r_1})
|\widehat{f}(\xi,\eta)|^2d\xi d\eta.
\end{split}
\end{equation}
Write $s_1=1/4+\rho/3$ and $r_1=1/2+2\rho/3$, where $\rho>0$. Thus,
\begin{equation}  \label{a8}
(1+|\xi|^{6s_1}+|\eta|^{3r_1})\leq c
(1+|\xi|^2+|\eta|^2)^{3/4+\rho}.
\end{equation}
Using \eqref{a8} in \eqref{a7} one easily shows the desired conclusion.\\
\end{proof}

Finally, we also recall the Leibniz rule for fractional derivatives.

\begin{lemma}   \label{lemmalei}
Let $0<\alpha<1$ and $1<p<\infty$. Then
$$
\|D^\alpha(fg)-fD^\alpha g -gD^\alpha f \|_{L^p(\mathbb{R})} \leq c
\|g\|_{L^\infty(\mathbb{R})} \|D^\alpha f \|_{L^p(\mathbb{R})},
$$
where $D^\alpha$ denotes either $D^\alpha_x$ or $D^\alpha_y.$
\end{lemma}
\begin{proof}
See Kenig, Ponce, and Vega \cite[Theorem A.12]{KPV}.\\
\end{proof}

\section{Proofs of Theorems \ref{theorem1}--\ref{globaltheorem}}
\label{local}

We begin this section by showing Theorem \ref{theorem1}. Since the
proof of Theorem \ref{theorem2} is similar, we only sketch it. We
finish the section by proving Theorem \ref{globaltheorem}.

\begin{proof}[Proof of Theorem \ref{theorem1}]
As usual, we consider the integral operator
\begin{equation}\label{Psi}
\Psi(u)(t)=\Psi_{u_0}(u)(t):=U(t)u_0+\int_0^t
U(t-t')(u^ku_x)(t')dt',
\end{equation}
and define the metric spaces
$$
\mathcal{Y}_T=\{ u \in C([0,T];H^s(\mathbb{R}^2)); \,\,\,\, \trinorm
u \trinorm <\infty \}
$$
and
$$
\mathcal{Y}_T^a=\{ u \in \mathcal{X}_T; \,\,\,\, \trinorm u \trinorm
\leq a \},
$$
with
$$
\trinorm u \trinorm:= \|u\|_{L^\infty_TH^s_{xy}}+
\|u\|_{L^{p_k}_TL^\infty_{xy}} + \|u_x\|_{L^{12/5}_TL^\infty_{xy}} +
\|D^s_xu_x\|_{L^\infty_xL^2_{yT}} +
\|D^s_yu_x\|_{L^\infty_xL^2_{yT}} + \|u\|_{L^4_xL^\infty_{yT}},
$$
where $a,T>0$ will be chosen later. We assume that $3/4<s<1$ and
$T\leq 1$.

First we estimate the $H^s$-norm of $\Psi(u)$. Let $u\in
\mathcal{Y}_T$. By using Minkowski's inequality, group properties
and then H\"older's inequality, we have
\begin{equation}\label{b5}
\begin{split}
 \|\Psi(u)(t) \|_{L^2_{xy}} & \leq  {\displaystyle c \|u_0\|_{H^s}+
 c\int_0^T \|u\|_{L^2_{xy}} \|u^{k-1}u_x\|_{L^\infty_{xy}} dt'  } \\
& \leq c \|u_0\|_{H^s}+ c  \|u\|_{L^\infty_TL^2_{xy}} \int_0^T
 \|u\|_{L^\infty_{xy}}^{k-1}\|u_x\|_{L^\infty_{xy}} dt' \\
& \leq c \|u_0\|_{H^s}+ c T^{\gamma} \|u\|_{L^\infty_TL^2_{xy}}
\|u_x\|_{L^{12/5}_TL^\infty_{xy}}\|u\|^{k-1}_{L^{p_k}_TL^\infty_{xy}}.
\end{split}
\end{equation}

Using group properties, Minkowski and H\"older's inequalities and
twice Lemma \ref{lemmalei}, we have
\begin{equation}   \label{b6}
\begin{split}
 \|D^s_x & \Psi(u)(t) \|_{L^2_{xy}}  \leq  {\displaystyle  \|D^s_xu_0\|_{L^2_{xy}}+\int_0^T \|D^s_x(u^ku_x)\|_{L^2_{xy}} dt'   }
  \\
 & \leq    {\displaystyle c \|u_0\|_{H^s}+ c\int_0^T
 \|u_x\|_{L^\infty_{xy}}  \|D^s_x(u^k)\|_{L^2_{xy}}dt' + c\int_0^T
 \|u^kD^s_xu_x\|_{L^2_{xy}}dt'   }  \\
   & \leq    {\displaystyle c \|u_0\|_{H^s}+ c\int_0^T
 \|u_x\|_{L^\infty_{xy}} \|u\|^{k-1}_{L^\infty_{xy}} \|D^s_xu\|_{L^2_{xy}}dt' + c\int_0^T
 \|u^kD^s_xu_x\|_{L^2_{xy}}dt'   }  \\
 & \leq   {\displaystyle c \|u_0\|_{H^s}+ c \|u\|_{L^\infty_TH^s_{xy}} \int_0^T
 \|u_x\|_{L^\infty_{xy}} \|u\|^{k-1}_{L^\infty_{xy}}dt'  + c\int_0^T
 \|u^kD^s_xu_x\|_{L^2_{xy}}dt'   }.
\end{split}
  \end{equation}
As in \eqref{b5}, from  H\"older's inequality, we get
\begin{equation}\label{b7}
\int_0^T
 \|u_x\|_{L^\infty_{xy}} \|u\|_{L^\infty_{xy}}^{k-1}dt' \leq c T^{\gamma}
\|u_x\|_{L^{12/5}_TL^\infty_{xy}}\|u\|^{k-1}_{L^{p_k}_TL^\infty_{xy}}.
\end{equation}
Moreover,
\begin{equation}\label{b8}
\begin{split}
{\displaystyle \int_0^T
 \|u^kD^s_xu_x\|_{L^2_{xy}}dt'}  & \leq {\displaystyle \int_0^T \|u\|^{k-2}_{L^\infty_{xy}}
 \|u^2D^s_xu_x\|_{L^2_{xy}}dt' }\\
&\leq {\displaystyle T^{\gamma}
\|u\|^{k-2}_{L^{\widetilde{p}_k}_TL^\infty_{xy}}
\|u\|^2_{L^4_xL^\infty_{yT}} \|D^s_xu_x\|_{L^\infty_xL^2_{yT}} },
\end{split}
\end{equation}
where $\widetilde{p}_k=\frac{2(k-2)}{1-2\gamma}$. Note that for
$3\leq k \leq7$ we have $\widetilde{p}_k<p_k$. Thus, combining
\eqref{b7}-\eqref{b8} with \eqref{b6}, we then deduce
\begin{eqnarray}   \label{b9}
\begin{array}{ccl}
{\displaystyle  \|D^s_x\Psi(u)(t) \|_{L^2_{xy}} }  & \leq &
{\displaystyle c \|u_0\|_{H^s}+ c
T^{\gamma}\|u\|_{L^\infty_TH^s_{xy}}
\|u_x\|_{L^{12/5}_TL^\infty_{xy}}\|u\|^{k-1}_{L^{p_k}_TL^\infty_{xy}}
}
  \\\\

& & + {\displaystyle  cT^{\gamma}
\|u\|^{k-2}_{L^{p_k}_TL^\infty_{xy}} \|u\|^2_{L^4_xL^\infty_{yT}}
\|D^s_xu_x\|_{L^\infty_xL^2_{yT}} }.
 \end{array}
  \end{eqnarray}
A similar analysis can be carried out to see that
\begin{eqnarray}   \label{b10}
\begin{array}{ccl}
{\displaystyle  \|D^s_y\Psi(u)(t) \|_{L^2_{xy}} }  & \leq &
{\displaystyle c \|u_0\|_{H^s}+ c
T^{\gamma}\|u\|_{L^\infty_TH^s_{xy}}
\|u_x\|_{L^{12/5}_TL^\infty_{xy}}\|u\|^{k-1}_{L^{p_k}_TL^\infty_{xy}}
}
  \\\\

& & + {\displaystyle  cT^{\gamma}
\|u\|^{k-2}_{L^{p_k}_TL^\infty_{xy}} \|u\|^2_{L^4_xL^\infty_{yT}}
\|D^s_yu_x\|_{L^\infty_xL^2_{yT}} }.
 \end{array}
  \end{eqnarray}
Therefore, from \eqref{b5}, \eqref{b9} and \eqref{b10}, we deduce
\begin{equation}\label{b11}
    \|\Psi(u)\|_{L^\infty_TH^s} \leq  c \|u_0\|_{H^s} +  cT^{\gamma} \trinorm u
    \trinorm^{k+1}.
\end{equation}

Next, we estimate the remaining norms. By taking $\delta=3/4$ and
$\theta=1/4+\sigma$, $0<\sigma\leq\frac{1-12\gamma}{24}$, in Lemma
\ref{lemma5}, we see that $p_k\leq3/\theta$. Thus, Lemma
\ref{lemma5}, group properties and the arguments used to obtain
\eqref{b11} yield
\begin{equation}   \label{b12}
\begin{split}
 \|\Psi(u) \|_{L^{p_k}_TL^\infty_{xy}} & \leq  {\displaystyle  \|U(t)u_0\|_{L^{p_k}_TL^\infty_{xy}}
 +\left\|U(t) \left( \int_0^t U(-t') (u^ku_x)(t') dt' \right) \right\|_{L^{p_k}_TL^\infty_{xy}}  }
  \\
 &\leq   {\displaystyle c \| u_0\|_{H^{3/4}}+
 c\int_0^T \|u^ku_x\|_{H^{3/4}} dt' } \\
  & \leq   {\displaystyle c \| u_0\|_{H^{s}}+
 c\int_0^T \|u^ku_x\|_{H^{s}} dt' } \\
& \leq  c \|u_0\|_{H^s} +  cT^{\gamma} \trinorm u
    \trinorm^{k+1}.
 \end{split}
  \end{equation}

By choosing  $\varepsilon \sim 1/2$ such that $1-\varepsilon/2 \leq
s$, an application of Lemma \ref{lemma2} together with arguments
similar to those ones used to derive \eqref{b11} imply
\begin{equation}   \label{b13}
\begin{split}
 \|\partial_x\Psi(u) \|_{L^{12/5}_TL^\infty_{xy}} & \leq  {\displaystyle  \|U(t)\partial_xu_0\|_{L^{12/5}_TL^\infty_{xy}}
 +\left\|U(t) \left( \int_0^t U(-t') \partial_x(u^ku_x)(t') dt' \right) \right\|_{L^{12/5}_TL^\infty_{xy}}  }
  \\
& \leq   {\displaystyle c \|D^{-\varepsilon/2}_x \partial_x
u_0\|_{L^2_{xy}}+
 c\int_0^T \|D^{-\varepsilon/2}_x \partial_x(u^ku_x)\|_{L^2_{xy}} dt'
 } \\
 & \leq   {\displaystyle c \|  u_0\|_{H^s}+
 c\int_0^T \| u^ku_x\|_{L^2_{xy}} dt'+
 c\int_0^T \|D^s_x (u^ku_x)\|_{L^2_{xy}} dt' }  \\
& \leq  c \|u_0\|_{H^s} + cT^{\gamma} \trinorm u
    \trinorm^{k+1}.
  \end{split}
  \end{equation}

Applying Lemma \ref{lemma1}, group properties, Minkowski and
H\"older inequalities, we obtain
\begin{equation}
\begin{split}
\|D^s_x\partial_x\Psi(u) \|_{L^{\infty}_xL^2_{yT}}\!\leq\!&
 \,\,\|\partial_xU(t)D^s_x u_0\|_{L^{\infty}_xL^2_{yT}}\\
\!& + \!\left\|\partial_xU(t) \left( \int_0^t\! U(-t')
D^s_x(u^ku_x)(t') dt'
\right) \right\|_{L^{\infty}_xL^2_{yT}}\\
\label{b14}\leq &\,\,{\displaystyle c \|D^s_x  u_0\|_{L^2_{xy}}+ c\int_0^T \|D^s_x (u^ku_x)\|_{L^2_{xy}} dt' }\\
 \leq &\,\, c \|u_0\|_{H^s} + cT^{\gamma} \trinorm u
    \trinorm^{k+1}
\end{split}
\end{equation}
and
\begin{equation}   \label{b15}
\begin{split}
\|D^s_y\partial_x\Psi(u) \|_{L^{\infty}_xL^2_{yT}}\! \leq & \,\,
\|\partial_xU(t)D^s_y u_0\|_{L^{\infty}_xL^2_{yT}}
 \\
\!&+\!\left\|\partial_xU(t) \left( \int_0^t U(-t') D^s_y(u^ku_x)(t') dt' \right) \right\|_{L^{\infty}_xL^2_{yT}}\\
\leq &\,\,{\displaystyle c \|D^s_y  u_0\|_{L^2_{xy}}+ c\int_0^T \|D^s_y (u^ku_x)\|_{L^2_{xy}} dt'}\\
\leq &\,\, c \|u_0\|_{H^s} +  cT^{\gamma} \trinorm u
    \trinorm^{k+1}.
\end{split}
\end{equation}
Finally, an application of Corollary \ref{corollary1}, Minkowski's
inequality, group properties, and arguments previously used yield
\begin{eqnarray}   \label{b16}
\begin{array}{ccl}
 \|\Psi(u) \|_{L^4_xL^\infty_{yT}}  & \leq &  {\displaystyle \|U(t) u_0\|_{L^4_x L^\infty_{yT}}
 +\left\| U(t) \left( \int_0^t U(-t')(u^ku_x)(t') dt' \right) \right\|_{L^4_xL^\infty_{yT}}  }
  \\\\

 & \leq & {\displaystyle c\|u_0\|_{H^s} + c\int_0^T \|u^ku_x\|_{H^s}dt'}  \\\\

 & \leq & c \|u_0\|_{H^s} +  cT^{\gamma} \trinorm u
    \trinorm^{k+1}.
  \end{array}
  \end{eqnarray}

Therefore, from \eqref{b11}--\eqref{b16}, we infer
$$
\trinorm \Psi(u)\trinorm \leq  c\|u_0\|_{H^s} + cT^{\gamma}
\trinorm u
    \trinorm^{k+1}.
$$
Choose $a=2c \|u_0\|_{H^s}$, and $T>0$ such that
$$
c\,a^k T^{\gamma} \leq \frac{1}{4}.
$$
Then, it is easy to see that $\Psi:\mathcal{Y}_T^a \mapsto
\mathcal{Y}_T^a$ is well defined. Moreover, similar arguments show
that $\Psi$ is a contraction. To finish the proof we use standard
arguments, thus, we omit the details. This completes the proof of
Theorem \ref{theorem1}.
\end{proof}

\begin{proof}[Proof of Theorem \ref{theorem2}]
The proof is very similar to that of Theorem \ref{theorem1}. So, we
give only the main steps. Assume $s_k<s<1$ and $0<T\leq1$. Define
the metric space
$$
\mathcal{X}_T=\{ u \in C([0,T];H^s(\mathbb{R}^2)); \,\,\,\, \trinorm
u \trinorm_{s,k} <\infty \}
$$
with
$$
\trinorm u \trinorm_{s,k}:= \|u\|_{L^\infty_TH^s_{xy}}+
\|u\|_{L^{\widetilde{p}_k}_TL^\infty_{xy}} +
\|u_x\|_{L^{12/5}_TL^\infty_{xy}} +
\|D^s_xu_x\|_{L^\infty_xL^2_{yT}} +
\|D^s_yu_x\|_{L^\infty_xL^2_{yT}} + \|u\|_{L^4_xL^\infty_{yT}}.
$$

We first note that since $k\geq8$ we have $\widetilde{p}_k>p_k$,
where $p_k=\frac{12(k-1)}{7-12\gamma}$ is given in Theorem
\ref{theorem1}. Hence, similarly to estimates
\eqref{b5}-\eqref{b10}, we establish that
\begin{equation}\label{b17}
    \|\Psi(u)\|_{L^\infty_TH^s} \leq  c \|u_0\|_{H^s} +  cT^{\gamma} \trinorm u
    \trinorm^{k+1}_{s,k},
\end{equation}
where $\Psi$ is the integral operator given in \eqref{Psi}. The
estimates \eqref{b13}-\eqref{b16} also hold here without any change.
What is left, is to show a similar estimate as \eqref{b12}. Here, to
use Lemma \ref{lemma5} we take $\delta=s$ and $\theta=1-s+\sigma$,
where $\sigma$ and $\gamma$ are chosen such that
\begin{equation}\label{b18}
s>s_k+\frac{6\gamma}{2(k-2)}+\sigma.
\end{equation}
The inequality \eqref{b18} promptly implies that
$\widetilde{p}_k\leq 3/\theta$. Thus, in view of Lemma \ref{lemma5},
we obtain
\begin{equation*}
\begin{split}
 \|\Psi(u) \|_{L^{\widetilde{p}_k}_TL^\infty_{xy}}
  & \leq   {\displaystyle c \| u_0\|_{H^{s}}+
 c\int_0^T \|u^ku_x\|_{H^{s}} dt' } \\
& \leq  c \|u_0\|_{H^s} +  cT^{\gamma} \trinorm u
    \trinorm^{k+1}_{s,k}.
 \end{split}
  \end{equation*}
Collecting all of our estimates, we then deduce
$$
\trinorm \Psi(u)\trinorm_{s,k} \leq  c\|u_0\|_{H^s} + cT^{\gamma}
\trinorm u \trinorm^{k+1}_{s,k}.
$$
The rest of the proof runs as before.
\end{proof}

\begin{proof}[Proof of Theorem \ref{ill-posedness}]
We start by recalling some facts about solitary wave  for the
generalized ZK equation. In fact, solitary wave  are special
solutions of the equation in \eqref{IVP} having the form
$u(x,y,t)=\ff_c(x-ct,y)$, for some $c\in \mathbb{R}$. Thus,
substituting this form of $u$ in \eqref{IVP} and integrating once,
we see that $\ff_c$ must satisfy
\begin{equation}  \label{solzk}
-c\ff_c+\Delta\ff_c+\frac{1}{k+1}\ff^{k+1}_c=0.
\end{equation}

The following lemma is well known and will be sufficient to
establish our result.

\begin{lemma}
Let $c>0$. Then equation \eqref{solzk} admits a positive, radially
symmetric  solution $\varphi_c\in H^1(\mathbb{R}^2)$. Moreover,
$\varphi_c\in C^\infty(\mathbb{R}^2)$, and there exists $\rho>0$
such that for all multi-index $\alpha\in\mathbb{N}^2$ with
$|\alpha|\leq2$, one has $|D^{\alpha}\ff_c(x)|\leq C_\alpha
e^{-\rho|x|}$, where $C_\alpha$ depends only on $\alpha$.
\end{lemma}
\begin{proof}
See Berestycki and Lions \cite{bl}.
\end{proof}

 It is easy to see that
$$
\ff_c(x,y)=c^{1/k}\ff_1\left(\sqrt{c}x,\sqrt{c}y\right),\qquad
\textrm{for all} \quad c>0,
$$
where $\ff_1$ is the solution of \eqref{solzk} with $c=1$. Thus,
since
\begin{equation}\label{fourierphi}
\widehat{\ff}_c(\xi,\eta)=c^{1/k-1}\widehat{\ff}_1\left(\frac{\xi}{\sqrt{c}},\frac{\eta}{\sqrt{c}}\right),
\end{equation}
one easily checks that
\begin{equation} \label{u1}
\begin{split}
\left\|\ff_c\right\|_{\dot{H}^{s_c(k)}}&=
c^{1/k-1/2+s_c(k)/2}\left\|\ff_1\right\|_{\dot{H}^{s_c(k)}}=\left\|\ff_1\right\|_{\dot{H}^{s_c(k)}}=:a_0
\end{split}
\end{equation}
Note that the constant $a_0$ does not depend on $c$.

Next, for any $c>0$ fixed, we consider
$$
u_c(x,y,t)=\varphi_c(x-ct,y).
$$
Hence, at $t=0$, we have $u_c(0)=\varphi_c$. Moreover, for any $c_1,
c_2>0$, we obtain
\begin{equation} \label{u2}
\left\|\ff_{c_1}-\ff_{c_2}\right\|_{\dot{H}^{s_c(k)}}^2
=\left\|\ff_{c_1}\right\|_{\dot{H}^{s_c(k)}}^2
+\|\ff_{c_2}\|_{\dot{H}^{s_c(k)}}^2-2
\langle\ff_{c_1},\ff_{c_2}\rangle_{\dot{H}^{s_c(k)}}.
\end{equation}
But, using \eqref{fourierphi} again, we obtain
 \begin{equation*}
 \begin{split}
\langle\ff_{c_1},\ff_{c_2}\rangle_{\dot{H}^{s_c(k)}}&=\int_{\rr^2}D^{s_c(k)}\ff_{c_1}(x,y)
\overline{D^{s_c(k)}\ff_{c_2}}(x,y)\;dxdy\\
&=\int_{\rr^2}|(\xi,\eta)|^{2s_c(k)}\;\widehat{\ff}_{c_1}(\xi,\eta)\;\overline{\widehat{\ff}_{c_2}}(\xi,\eta)\;d\xi d\eta\\
&=(c_1c_2)^{\frac{1}{k}-1}\int_{\rr^2}|(\xi,\eta)|^{2s_c(k)}
\widehat{\ff}_1\left(\frac{\xi}{\sqrt{c_1}},\frac{\eta}{\sqrt{c_1}}\right)
\overline{\widehat{\ff}_1}\left(\frac{\xi}{\sqrt{c_2}},\frac{\eta}{\sqrt{c_2}}\right)\;d\xi d\eta\\
&=\left(\frac{c_2}{c_1}\right)^{\frac{1}{k}-1}\int_{\rr^2}|(\xi,\eta)|^{2s_c(k)}\;\widehat{\ff}_1(\xi,\eta)\;
\overline{\widehat{\ff}_1}\left(\sqrt{\frac{c_1}{c_2}}\xi,\sqrt{\frac{c_1}{c_2}}\eta\right)\;d\xi
d\eta.
 \end{split}
 \end{equation*}
Therefore, as $\theta:=c_1/c_2\to1$, we get
\begin{equation}  \label{u3}
\lim_{\theta\to1}\langle\ff_{c_1},\ff_{c_2}\rangle_{s_c(k)}=a_0^2.
\end{equation}
As a consequence of \eqref{u1}-\eqref{u3}, we then get
 \[
\lim_{\theta\to1}\left\|\ff_{c_1}-\ff_{c_2}\right\|_{\dot{H}^{s_c(k)}}=0.
 \]

On the other hand, for any $t>0$,
\[\begin{split}
\left\|u_{c_1}(t)-u_{c_2}(t)\right\|_{\dot{H}^{s_c(k)}}^2
=\left\|u_{c_1}(t)\right\|_{\dot{H}^{s_c(k)}}^2
+\left\|u_{c_2}(t)\right\|_{\dot{H}^{s_c(k)}}^2-2 \langle
u_{c_1}(t),u_{c_2}(t)\rangle_{\dot{H}^{s_c(k)}}.
\end{split}\]
But, since
$$
\widehat{u_c(t)}(\xi,\eta)=c^{1/k-1}e^{-ic\xi
t}\widehat{\varphi}_1\left(\frac{\xi}{\sqrt{c}},\frac{\eta}{\sqrt{c}}\right),
$$
we deduce
 \[
 \begin{split}
\langle
&u_{c_1}(t),u_{c_2}(t)\rangle_{\dot{H}^{s_c(k)}}\\
&=(c_1c_2)^{\frac{1}{k}-1}\int_{\rr^2}e^{-it\xi(c_1-c_2)}|(\xi,\eta)|^{2s_c(k)}
\widehat{\ff_1}\left(\frac{\xi}{\sqrt{c_1}},\frac{\eta}{\sqrt{c_1}}\right)\;\overline{\widehat{\ff_1}}
\left(\frac{\xi}{\sqrt{c_2}},\frac{\eta}{\sqrt{c_2}}\right)\;d\xi d\eta\\
&=\left(\frac{c_2}{c_1}\right)^{\frac{1}{k}-1}\int_{\rr^2}e^{-it\xi
\sqrt{c_1}(c_1-c_2)}|(\xi,\eta)|^{2s_c(k)}\widehat{\ff_1}(\xi,\eta)\;
\overline{\widehat{\ff_1}}\left(\sqrt{\frac{c_1}{c_2}}\xi,\sqrt{\frac{c_1}{c_2}}\eta\right)\;d\xi
d\eta.
 \end{split}
 \]
By choosing $c_1=m+1$ and $c_2=m\in\N$ and letting $m\to\infty$, an
application of the Riemann-Lebesgue lemma, yields
\[
\lim_{m\to\infty}\langle
u_{c_1}(t),u_{c_2}(t)\rangle_{\dot{H}^{s_c(k)}}=0.
\]
Therefore, for any $t>0$,
\[
\lim_{\theta\to1}\left\|u_{c_1}(t)-u_{c_2}(t)\right\|_{\dot{H}^{s_c(k)}}=\sqrt{2}\;a_0.
 \]
This completes the proof of the theorem.

\end{proof}

\begin{proof}[Proof of Theorem \ref{globaltheorem}]
By using the Gagliardo-Nirenberg interpolation theorem it follows
that
\begin{equation}  \label{gn}
\|u(t)\|_{L^{k+2}}^{k+2}\leq c
\|u(t)\|_{L^2}^2\|\partial_xu(t)\|_{L^2}^k.
\end{equation}
Combining \eqref{mass}, \eqref{energy} and \eqref{gn}, we obtain
that
\begin{equation*}
\begin{split}
\|u(t)\|_{H^1}^2&=I_1(u(t))+I_2(u(t))+c\|u(t)\|_{L^{k+2}}^{k+2}\\
&\leq I_1(u_0)+I_2(u_0)+c\|u_0\|^2_{L^2}\|\partial_xu(t)\|_{L^2}^k.
\end{split}
\end{equation*}
Denote $X(t)=\|u(t)\|_{H^1}^2$. Since $k\geq3$, we then have
$$
X(t)\leq C(\|u_0\|_{H^1})+c\|u_0\|^2_{L^2}X(t)^{1+\frac{k-2}{2}}.
$$
Thus, if $\|u_0\|_{H^1}$ is small enough, a standard argument leads
to $\|u(t)\|_{H^1}\leq C(\|u_0\|_{H^1})$ for $t\in[0,T]$. Therefore,
we can apply the local theory to extend the solution.
\end{proof}

\section{Global well-posedness for the modified ZK} \label{global}

In this section, we consider the Cauchy problem associated with the
modified ZK. The main goal is to prove the global well-posedness
result stated in Theorem \ref{globalm}.

\subsection{Auxiliary results}

We start with the the following local well-posedness result.
The proof is slightly different from that of Theorem 1.1 in
\cite{LP}.

\begin{theorem}   \label{theorem3}
Let $k=2$. For any $u_0 \in H^s(\mathbb{R}^2)$, $s>3/4$, there exist
$T=T(\|u_0\|_{H^s})>0$ and a unique solution of the IVP \eqref{IVP},
defined in the interval $[0,T]$, such that
\begin{equation}\label{c1}
u \in C([0,T];H^s(\mathbb{R}^2)),
\end{equation}
\begin{equation}\label{c2}
\|D^s_xu_x\|_{L^\infty_xL^2_{yT}}  +
\|D^s_yu_x\|_{L^\infty_xL^2_{yT}}  <\infty,
\end{equation}
\begin{equation}\label{c3}
 \|u\|_{L^{p}_T L^\infty_{xy}}+   \|u_x\|_{L^{12/5}_T L^\infty_{xy}}
  <\infty,
\end{equation}
and
\begin{equation}\label{c4}
\|u\|_{L^2_x L^\infty_{yT}}
  <\infty,
\end{equation}
where $p=\frac{2}{1-2\gamma}$ and $\gamma\in (0,5/12)$. In addition,
the following statements hold:

\begin{itemize}
  \item[(i)] For any $T'\in(0,T)$ there exists a neighborhood $V$ of $u_0$ in
$H^s(\mathbb{R}^2)$ such that the map $\widetilde{u}_0\mapsto
\widetilde{u}(t)$ from $V$ into the class defined by
\eqref{c1}--\eqref{c4} is smooth.

  \item[(ii)] The existence time $T$ is given by
\begin{equation}\label{time}
T\sim \|u_0\|_{H^s}^{-2/\gamma}.\\
\end{equation}
\end{itemize}
\end{theorem}

To simplify the exposition and for further references, we prove first
the following lemma.

\begin{lemma}  \label{lemma6}
Assume $u,v,w$ are sufficiently smooth. Let $p$ be as in Theorem
\ref{theorem3}.
\begin{itemize}
  \item[(i)] For any $T>0$,
  $$\int_0^T \|vwu_x\|_{L^2_{xy}}dt'\leq c T^\gamma
\|u_x\|_{L^{12/5}_T L^\infty_{xy}}\|v\|_{L^{\infty}_T L^2_{xy}}
\|w\|_{L^{p}_T L^\infty_{xy}}.
$$

  \item[(ii)] For any $T>0$ and $s\in(0,1)$,
\begin{equation*}
\begin{split}
\int_0^T \|D^s_x(vwu_x)\|_{L^2_{xy}}dt'& \leq c T^\gamma
\|u_x\|_{L^{12/5}_T L^\infty_{xy}} \left\{ \|v\|_{L^{\infty}_T
H^s_{xy}} \|w\|_{L^{p}_T L^\infty_{xy}}+ \|w\|_{L^{\infty}_T H^s_{xy}} \|v\|_{L^{p}_T L^\infty_{xy}}\right\}\\
&\;\;\;\;+cT^\gamma \|w\|_{L^{p}_T L^\infty_{xy}} \|v\|_{L^2_x
L^\infty_{yT}}\|D^s_xu_x\|_{L^\infty_xL^2_{yT}}.
\end{split}
\end{equation*}
The same still holds if we replace $D^s_x$ by $D^s_y$.
\end{itemize}
\end{lemma}
\begin{proof}
The estimate (i) follows after applying H\"older
inequality. The proof of (ii) is  roughly an application of
Lemma \ref{lemmalei} combined with the H\"older inequality. Indeed,
applying twice Lemma \ref{lemmalei}, we deduce
\begin{equation}\label{c5}
\|D^s_x(vwu_x)\|_{L^2_{xy}}\leq
c\|D^s_xw\|_{L^2_{xy}}\|v\|_{L^\infty_{xy}}\|u_x\|_{L^\infty_{xy}}+
c\|D^s_xv\|_{L^2_{xy}}\|w\|_{L^\infty_{xy}}\|u_x\|_{L^\infty_{xy}}+
c\|wvD^s_xu_x\|_{L^2_{xy}}.
\end{equation}
For the first two terms, from H\"older's inequality, we obtain
\begin{equation*}
\begin{split}
\int_0^T\{\|D^s_xw\|_{L^2_{xy}}&\|v\|_{L^\infty_{xy}}\|u_x\|_{L^\infty_{xy}}+
\|D^s_xv\|_{L^2_{xy}}\|w\|_{L^\infty_{xy}}\|u_x\|_{L^\infty_{xy}}
\}dt'\\
&\leq c T^\gamma \|u_x\|_{L^{12/5}_T L^\infty_{xy}} \left\{
\|v\|_{L^{\infty}_T H^s_{xy}} \|w\|_{L^{p}_T L^\infty_{xy}}+
\|w\|_{L^{\infty}_T H^s_{xy}} \|v\|_{L^{p}_T L^\infty_{xy}}\right\}.
\end{split}
\end{equation*}
For the last term in \eqref{c5}, we obtain
\begin{equation*}
\begin{split}
\int_0^T\|wvD^s_xu_x\|_{L^2_{xy}}dt'&\leq \int_0^T
\|w\|_{L^\infty_{xy}}\|vD^s_xu_x\|_{L^2_{xy}} dt'\\
&\leq \left(\int_0^T\|w\|_{L^\infty_{xy}}^2dt'\right)^{1/2}\|vD^s_xu_x\|_{L^2_{xyT}}\\
&\leq cT^\gamma \|w\|_{L^{p}_T L^\infty_{xy}} \|v\|_{L^2_x
L^\infty_{yT}}\|D^s_xu_x\|_{L^\infty_xL^2_{yT}}.
\end{split}
\end{equation*}
This completes the proof of the lemma.\\
\end{proof}

\begin{proof}[Sketch of proof of Theorem \ref{theorem3}]
The proof is similarly carried out as the proof of Theorem
\ref{theorem1} (see also \cite{LP}). The main difference is
that instead of using the maximal function in Proposition
\ref{proposition2}(i), we use the one in (ii).

Thus, we consider the integral operator
$$
\Phi(u)(t)=\Phi_{u_0}(u)(t):=U(t)u_0+\int_0^t
U(t-t')(u^2u_x)(t')dt',
$$
and define the metric spaces
$$
\mathcal{Z}_T=\{ u \in C([0,T];H^s(\mathbb{R}^2)); \,\,\,\, \trinorm
u \trinorm_{s,2} <\infty \}
$$
and
$$
\mathcal{Z}_T^a=\{ u \in \mathcal{X}_T; \,\,\,\, \trinorm u
\trinorm_{s,2} \leq a \},
$$
with
$$
\trinorm u \trinorm_{s,2}:= \|u\|_{L^\infty_TH^s_{xy}}+
\|u\|_{L^{p}_TL^\infty_{xy}} + \|u_x\|_{L^{12/5}_TL^\infty_{xy}} +
\|D^s_xu_x\|_{L^\infty_xL^2_{yT}} +
\|D^s_yu_x\|_{L^\infty_xL^2_{yT}} + \|u\|_{L^2_xL^\infty_{yT}},
$$
where $a,T>0$ will be chosen later. We assume that $3/4<s<1$ and
$T\leq 1$.

Here, we only estimate the ${L^\infty_TH^s_{xy}}$-norm, because the
others ones are obtained as in Theorem \ref{theorem1}. From group
properties, Minkowski's inequality, and Lemma \ref{lemma6} it follows that
\begin{equation}\label{c6}
\begin{split}
 \|  \Phi(u)(t) \|_{H^s_{xy}}  &\leq  {\displaystyle  \|u_0\|_{H^s_{xy}}+\int_0^T \|u^2u_x\|_{H^s_{xy}} dt'   }
  \\
& \leq \|u_0\|_{H^s_{xy}}+c T^\gamma \|u_x\|_{L^{12/5}_T
L^\infty_{xy}} \left\{ \|u\|_{L^{\infty}_T H^s_{xy}} \|u\|_{L^{p}_T
L^\infty_{xy}}+ \|u\|_{L^{\infty}_T H^s_{xy}} \|u\|_{L^{p}_T
L^\infty_{xy}}\right\}\\
&\;\;+cT^\gamma \|u\|_{L^{p}_T L^\infty_{xy}} \|u\|_{L^2_x
L^\infty_{yT}}\|D^s_xu_x\|_{L^\infty_xL^2_{yT}}+cT^\gamma
\|u\|_{L^{p}_T L^\infty_{xy}} \|u\|_{L^2_x
L^\infty_{yT}}\|D^s_yu_x\|_{L^\infty_xL^2_{yT}}
\end{split}
  \end{equation}
Thus,
$$
 \| \Phi(u) \|_{L^\infty_TH^s_{xy}} \leq
 \|u_0\|_{H^s_{xy}}+cT^\gamma\trinorm u \trinorm_{s,2}^3.
 $$
Finally, gathering together all estimates we see that
$$
 \trinorm \Phi(u) \trinorm_{s,2}\leq
c \|u_0\|_{H^s_{xy}}+cT^\gamma\trinorm u \trinorm_{s,2}^3.
 $$
Choosing $a=2c\; \|u_0\|_{H^s}$, and then $T$  such that
\begin{equation}\label{size}
cT^\gamma a^2<\frac{1}{20},
\end{equation}
we deduce that $\Phi:\mathcal{Z}_T^a\rightarrow\mathcal{Z}_T^a$ is
well defined and is a contraction. The rest of the proof follows
standard arguments. So we will omit it.
\end{proof}

\begin{proposition} \label{proposition3}
Consider the IVP
\begin{equation}\label{mIVP1}
     \left\{
\begin{array}{lll}
{\displaystyle v_t+\partial_x \Delta v+v^2v_x  =  0,  }  \qquad (x,y) \in \mathbb{R}^2, \,\,\,\, t>0 .\\
{\displaystyle  v(x,y,0)=v_0(x,y)}\in H^1(\mathbb{R}^2),
\end{array}
\right.
\end{equation}
Let $T\sim\|v_0\|_{H^1}^{-2/\gamma}$ be the existence time given by
Theorem \ref{theorem3}. If $v_0$ satisfies
\begin{equation}\label{c7}
\|v_0\|_{L^2}<\sqrt3\|\varphi\|_{L^2}\quad and \quad
\|v_0\|_{H^1}\sim N^{1-s},
\end{equation}
where $\varphi$ is the ground state solution of \eqref{solwave},
then
\begin{itemize}
  \item[(i)] the solution $v$ of \eqref{mIVP1} satisfies
  \begin{equation}\label{c8}
    \sup_{[0,T]}\|v(t)\|_{H^1}\leq c N^{1-s}.
  \end{equation}
\item[(ii)] For any $\rho\in(3/4,1)$, the solution $v$ of \eqref{mIVP1} satisfies
  \begin{equation}\label{c9}
    \trinorm v \trinorm_{\rho,2}\sim N^{\rho(1-s)}.
  \end{equation}
\end{itemize}
\end{proposition}
\begin{proof}
The proof of \eqref{c8} is similar to the proof of Theorem
\ref{globaltheorem}, but instead of using \eqref{gn}, we use the
following (sharp) Gagliardo-Nirenberg inequality (see \cite{We})
$$
\frac{1}{6} \|u(t)\|_{L^4}^4 \leq \frac{1}{3}
\left(\frac{\|u(t)\|_{L^2}}{\|\varphi\|_{L^2}}\right)^2 \|\nabla
u(t)\|_{L^2}^2.
$$
The proof of \eqref{c9} follows immediately from the proof of
Theorem \ref{theorem3} and the inequality $\|v_0\|_{H^\rho}\leq c
N^{\rho(1-s)}$.\\
\end{proof}

\begin{proposition}\label{proposition4}
Let $v_0\in H^1(\mathbb{R}^2)$ and $w_0\in H^\rho(\mathbb{R}^2)$,
$\rho>3/4$, and let $v$ be the solution given in Proposition
\ref{proposition3}. Then there exists a unique solution $w$ of the
IVP
\begin{equation}\label{mIVP2}
     \left\{
\begin{array}{lll}
{\displaystyle w_t+\partial_x \Delta w+w^2w_x +2wvv_x+2wvw_x+v^2w_x+w^2v_x =  0,  }  \qquad (x,y) \in \mathbb{R}^2, \,\,\,\, t>0, \\
{\displaystyle  w(x,y,0)=w_0(x,y)},
\end{array}
\right.
\end{equation}
defined in the same interval of existence of $v$, $[0,T]$, such that
\begin{equation}\label{c10}
w \in C([0,T];H^\rho(\mathbb{R}^2)),
\end{equation}
\begin{equation}\label{c11}
\|D^\rho_xw_x\|_{L^\infty_xL^2_{yT}}  +
\|D^\rho_yw_x\|_{L^\infty_xL^2_{yT}}  <\infty,
\end{equation}
\begin{equation}\label{c12}
 \|w\|_{L^{p}_T L^\infty_{xy}}+   \|w_x\|_{L^{12/5}_T L^\infty_{xy}}
  <\infty,
\end{equation}
and
\begin{equation}\label{c13}
\|w\|_{L^2_x L^\infty_{yT}}
  <\infty,
\end{equation}
where $p=\frac{2}{1-2\gamma}$ and $\gamma\in (0,5/12)$.
\end{proposition}
\begin{proof}[Sketch of the proof]
Define the integral operator
$$
\widetilde{\Phi}(w)(t)=\widetilde{\Phi}_{w_0}(w)(t):=U(t)w_0+\int_0^t
U(t-t')F(t')dt'
$$
where
 \begin{equation} \label{F}
F=w^2w_x+2wvv_x+2wvw_x+v^2w_x+w^2v_x.
\end{equation}
and the function metric spaces
$$
\mathcal{W}_T=\{ w \in C([0,T];H^\rho(\mathbb{R}^2)); \,\,\,\,
\trinorm w \trinorm_{\rho,2} <\infty \}
$$
and
$$
\mathcal{W}_T^a=\{ w \in \mathcal{W}_T; \,\,\,\, \trinorm w
\trinorm_{\rho,2} \leq a \},
$$
with
$$
\trinorm w \trinorm_{\rho,2}:= \|w\|_{L^\infty_TH^\rho_{xy}}+
\|w\|_{L^{p}_TL^\infty_{xy}} + \|w_x\|_{L^{12/5}_TL^\infty_{xy}} +
\|D^\rho_xw_x\|_{L^\infty_xL^2_{yT}} +
\|D^\rho_yw_x\|_{L^\infty_xL^2_{yT}} + \|w\|_{L^2_xL^\infty_{yT}},
$$
 As before, we only estimate the $L^\infty_T H^\rho_{xy}$-norm, because from our
linear estimates, all the others estimates reduce to this one.

First, we note that
$$
\|D^\rho_x\widetilde{\Phi}(w)\|_{L^2} \leq
\|w_0\|_{H^\rho}+\int_0^T\|D^\rho_xF\|_{L^2}dt'.
$$
But, successive applications of Lemma \ref{lemma6}(ii) lead to
\begin{equation*}
\begin{split}
\int_0^T \|D^\rho_x(w^2w_x)\|_{L^2_{xy}}dt'& \leq c T^\gamma
\|w_x\|_{L^{12/5}_T L^\infty_{xy}} \|w\|_{L^{\infty}_T H^\rho_{xy}} \|w\|_{L^{p}_T L^\infty_{xy}}\\
&\;\;\;\;+cT^\gamma \|w\|_{L^{p}_T L^\infty_{xy}} \|w\|_{L^2_x
L^\infty_{yT}}\|D^\rho_xw_x\|_{L^\infty_xL^2_{yT}},
\end{split}
\end{equation*}

\begin{equation*}
\begin{split}
\int_0^T \|D^\rho_x(wvv_x)\|_{L^2_{xy}}dt'& \leq c T^\gamma
\|v_x\|_{L^{12/5}_T L^\infty_{xy}} \left\{ \|v\|_{L^{\infty}_T
H^\rho_{xy}} \|w\|_{L^{p}_T L^\infty_{xy}}+ \|w\|_{L^{\infty}_T H^\rho_{xy}} \|v\|_{L^{p}_T L^\infty_{xy}}\right\}\\
&\;\;\;\;+cT^\gamma \|w\|_{L^{p}_T L^\infty_{xy}} \|v\|_{L^2_x
L^\infty_{yT}}\|D^\rho_xv_x\|_{L^\infty_xL^2_{yT}},
\end{split}
\end{equation*}

\begin{equation*}
\begin{split}
\int_0^T \|D^\rho_x(wvw_x)\|_{L^2_{xy}}dt'& \leq c T^\gamma
\|w_x\|_{L^{12/5}_T L^\infty_{xy}} \left\{ \|v\|_{L^{\infty}_T
H^\rho_{xy}} \|w\|_{L^{p}_T L^\infty_{xy}}+ \|w\|_{L^{\infty}_T H^\rho_{xy}} \|v\|_{L^{p}_T L^\infty_{xy}}\right\}\\
&\;\;\;\;+cT^\gamma \|w\|_{L^{p}_T L^\infty_{xy}} \|v\|_{L^2_x
L^\infty_{yT}}\|D^\rho_xw_x\|_{L^\infty_xL^2_{yT}},
\end{split}
\end{equation*}

\begin{equation*}
\begin{split}
\int_0^T \|D^\rho_x(v^2w_x)\|_{L^2_{xy}}dt'& \leq c T^\gamma
\|w_x\|_{L^{12/5}_T L^\infty_{xy}} \|v\|_{L^{\infty}_T
H^\rho_{xy}} \|v\|_{L^{p}_T L^\infty_{xy}}\\
&\;\;\;\;+cT^\gamma \|v\|_{L^{p}_T L^\infty_{xy}} \|v\|_{L^2_x
L^\infty_{yT}}\|D^\rho_xw_x\|_{L^\infty_xL^2_{yT}},
\end{split}
\end{equation*}

\begin{equation*}
\begin{split}
\int_0^T \|D^\rho_x(w^2v_x)\|_{L^2_{xy}}dt'& \leq c T^\gamma
\|v_x\|_{L^{12/5}_T L^\infty_{xy}} \|w\|_{L^{\infty}_T
H^\rho_{xy}} \|w\|_{L^{p}_T L^\infty_{xy}}\\
&\;\;\;\;+cT^\gamma \|w\|_{L^{p}_T L^\infty_{xy}} \|w\|_{L^2_x
L^\infty_{yT}}\|D^\rho_xv_x\|_{L^\infty_xL^2_{yT}}.
\end{split}
\end{equation*}
Thus, we see that
$$
\|D^\rho_x\widetilde{\Phi}(w)\|_{L^2} \leq
\|w_0\|_{H^\rho}+cT^\gamma\left\{\trinorm w \trinorm_{\rho,2}^2 +
\trinorm v \trinorm_{1,2}^2 +\trinorm w \trinorm_{\rho,2}\trinorm v
\trinorm_{1,2}\right\} \trinorm w \trinorm_{\rho,2}.
$$
Analogously, we deduce
$$
\|D^\rho_y\widetilde{\Phi}(w)\|_{L^2} \leq
\|w_0\|_{H^\rho}+cT^\gamma\left\{\trinorm w \trinorm_{\rho,2}^2 +
\trinorm v \trinorm_{1,2}^2 +\trinorm w \trinorm_{\rho,2}\trinorm v
\trinorm_{1,2}\right\} \trinorm w \trinorm_{\rho,2}
$$
and
$$
\|\widetilde{\Phi}(w)\|_{L^2} \leq
\|w_0\|_{L^2}+cT^\gamma\left\{\trinorm w \trinorm_{\rho,2}^2 +
\trinorm v \trinorm_{1,2}^2 +\trinorm w \trinorm_{\rho,2}\trinorm v
\trinorm_{1,2}\right\} \trinorm w \trinorm_{\rho,2}.
$$
Therefore, we have established that
\begin{equation}\label{c14}
\trinorm\widetilde{\Phi}(w)\trinorm_{\rho,2}
\leq \|w_0\|_{H^\rho}+cT^\gamma\left\{\trinorm w \trinorm_{\rho,2}^2
+ \trinorm v \trinorm_{1,2}^2 +\trinorm w \trinorm_{\rho,2}\trinorm
v \trinorm_{1,2}\right\} \trinorm w \trinorm_{\rho,2}.
\end{equation}
Now, by choosing $a=2c\;\max\{\|v_0\|_{H^1}, \|w_0\|_{H^\rho}\}$, we
see that
\begin{equation}\label{size2}
cT^\gamma a^2<\frac{1}{20}.
\end{equation}
As a consequence, $\widetilde{\Phi}:
\mathcal{W}_T^a\rightarrow\mathcal{W}_T^a$ is well defined. To
finish the proof, one proceeds as usual. This proves the theorem. \\
\end{proof}

\begin{corollary}\label{corollary6}
Let $s\in(3/4,1)$ be fixed. Let $v_0\in H^1(\mathbb{R}^2)$ and
$w_0\in H^\rho(\mathbb{R}^2)$ with $3/4<\rho\leq s$. Assume that the
initial data $w_0$ satisfies $\|w_0\|_{H^\rho}\sim N^{\rho-s}$ and
let $v$ and $w$ be the corresponding solutions given in Propositions
\ref{proposition3} and \ref{proposition4}, respectively. Then
$$
\trinorm w \trinorm_{\rho,2}\leq cN^{\rho-s}.
$$
\end{corollary}
\begin{proof}
From the proof of Proposition \ref{proposition4}, we have
\begin{equation} \label{c15}
w(t)=U(t)w_0+\int_0^t U(t-t')F(t')dt',\qquad t\in[0,T].
\end{equation}
Moreover, since $\widetilde{\Phi}:
\mathcal{W}_T^a\rightarrow\mathcal{W}_T^a$, we obtain $\trinorm w
\trinorm_{\rho,2}\leq a$, where $a=2c\;\max\{\|v_0\|_{H^1},
\|w_0\|_{H^\rho}\}$. Analogously, $\trinorm v \trinorm_{1,2}\leq a$.
Hence, from \eqref{c14} and \eqref{size}, we get
$$
\trinorm w \trinorm_{\rho,2}\leq \|w_0\|_{H^\rho}+\frac{3}{20}
\trinorm w \trinorm_{\rho,2}.
$$
This completes the proof of the lemma.\\
\end{proof}

\begin{lemma}   \label{lemma7}
Define
$$
\trinorm w
\trinorm_0:=\|w\|_{L^\infty_TL^2_{xy}}+\|w_x\|_{L^\infty_xL^2_{yT}}.
$$
 Let $v_0\in H^1(\mathbb{R}^2)$ and $w_0\in
H^s(\mathbb{R}^2)$, $3/4<s<1$, such that $\|v_0\|_{H^1}\sim N^{1-s}$
and $\|w_0\|_{L^2}\sim N^{-s}$. Let $v$ and $w$ be the solutions
given in Propositions \ref{proposition3} and \ref{proposition4},
respectively. Then,
$$
\trinorm w \trinorm_0\leq cN^{-s}.
$$
\end{lemma}
\begin{proof}
It follows from \eqref{c15} and Lemma \ref{lemma1} that
$$
\|w_x\|_{L^\infty_xL^2_{yT}}\leq
c\|w_0\|_{L^2}+c\int_0^T\|F(t')\|_{L^2}dt'.
$$
Now, applying Lemma \ref{lemma6}(i), we deduce
\begin{equation*}
\begin{split}
\int_0^T\|&F(t')\|_{L^2}dt'\leq c T^\gamma\Big\{
\|w\|_{L^p_TL^\infty_{xy}}\|w_x\|_{L^{12/5}_TL^\infty_{xy}}+\|v\|_{L^p_TL^\infty_{xy}}\|v_x\|_{L^{12/5}_TL^\infty_{xy}}\\
&\;\;+\|v\|_{L^p_TL^\infty_{xy}}\|w_x\|_{L^{12/5}_TL^\infty_{xy}}+\|w\|_{L^p_TL^\infty_{xy}}\|v_x\|_{L^{12/5}_TL^\infty_{xy}}
+\|v\|_{L^p_TL^\infty_{xy}}
\|v_x\|_{L^{12/5}_TL^\infty_{xy}}\Big\}\trinorm w \trinorm_0.
\end{split}
\end{equation*}
Hence, as in Corollary \ref{corollary6}, we obtain
$$
\|w_x\|_{L^\infty_xL^2_{yT}}\leq c\|w_0\|_{L^2}+\frac{1}{4}\trinorm
w \trinorm_0.
$$
Similarly, we have
$$
\|w\|_{L^\infty_TL^2_{xy}}\leq c\|w_0\|_{L^2}+\frac{1}{4}\trinorm w
\trinorm_0.
$$
This proves the lemma.\\
\end{proof}

\begin{proposition}\label{proposition5}
Assume that $w_0$ and $v_0$ satisfy the hypotheses of Corollary
\ref{corollary6} and Lemma \ref{lemma7}. Let $v$ and $w$ be the
solutions given in Propositions \ref{proposition3} and
\ref{proposition4}, respectively. Define
\begin{equation}\label{z}
z(t)=\int_0^tU(t-t')F(t')dt',
\end{equation}
where $F$ is given in \eqref{F}. Then,
$$
\|z\|_{L^\infty_TH^1}\sim N^{\frac{3-5s}{2}}.
$$
\end{proposition}
\begin{proof}
We begin by estimating $\|\partial_x z\|_{L^2_{xy}}$. The main tool
here is the estimate \eqref{dseffect}. Indeed,
\begin{equation*}
\begin{split}
\|\partial_x z\|_{L^2{xy}}&\leq
\|\partial_x\int_0^tU(t-t')F(t')dt'\|_{L^2_{xy}} \\
&\leq
c\left\{\|w^2w_x\|_{L^1_xL^2_{yT}}+\|wvv_x\|_{L^1_xL^2_{yT}}+\|wvw_x\|_{L^1_xL^2_{yT}}
+\|v^2w_x\|_{L^1_xL^2_{yT}}+\|w^2v_x\|_{L^1_xL^2_{yT}}\right\}\\
&=A_1+A_2+A_3+A_4+A_5.
\end{split}
\end{equation*}
Now, from Holder's inequality, we obtain
\begin{equation*}
\begin{split}
A_1\leq c
\|w\|^2_{L^2_xL^\infty_{yT}}\|w_x\|_{L^\infty_xL^2_{yT}}\leq c
\trinorm w\trinorm^2_{\rho,2}\trinorm w \trinorm_0.
\end{split}
\end{equation*}
Applying Corollary \ref{corollary6} and Lemma \ref{lemma7}, we then
get
$$
A_1\leq c N^{-s}\leq c N^{\frac{3-5s}{2}}.
$$
From this point on we apply several times  H\"older's inequality
without mentioning it.
$$
A_2\leq c
T^{1/12}\|v\|_{L^2_xL^\infty_{yT}}\|v_x\|_{L^{12/5}_TL^\infty_{xy}}\|w\|_{L^\infty_TL^2_{xy}}
\leq c T^{1/12}\trinorm v \trinorm^2_{\rho,2}\trinorm w \trinorm_0
$$
Thus, Proposition \ref{proposition3}(ii) and Lemma \ref{lemma7}
yield $A_2\sim N^{\frac{3-5s}{2}}$. To estimate $A_3$, we note that
$$
A_3\leq c
T^{1/12}\|v\|_{L^2_xL^\infty_{yT}}\|w_x\|_{L^{12/5}_TL^\infty_{xy}}\|w\|_{L^\infty_TL^2_{xy}}
\leq c T^{1/12}\trinorm v \trinorm_{\rho,2}\trinorm w
\trinorm_{\rho,2}\trinorm w \trinorm_0.
$$
Since $T\sim N^{-2(1-s)/\gamma}$, $\gamma\in(0,5/12)$, we deduce
that $T^{1/12}\sim N^{-2(1-s)/5}$. Hence, from Proposition
\ref{proposition3}, Corollary \ref{corollary6}, and Lemma
\ref{lemma7}, we infer  $A_3\sim N^{\frac{3-5s}{2}}$. Similarly,
since
$$
A_4\leq c
T^{1/12}\|w\|_{L^2_xL^\infty_{yT}}\|v_x\|_{L^{12/5}_TL^\infty_{xy}}\|w\|_{L^\infty_TL^2_{xy}}
\leq c T^{1/12}\trinorm w \trinorm_{\rho,2}\trinorm v
\trinorm_{\rho,2}\trinorm w \trinorm_0,
$$
we get $A_4\sim N^{\frac{3-5s}{2}}$.

Finally,
$$
A_5\leq c
\|v\|^2_{L^2_xL^\infty_{yT}}\|w_x\|_{L^\infty_xL^2_{yT}}\leq c
\trinorm v\trinorm^2_{\rho,2}\trinorm w \trinorm_0.
$$
Therefore,  Proposition \ref{proposition3} and Lemma \ref{lemma7}
yield $A_5\sim N^{\frac{3-5s}{2}}$. The same analysis can be
performed to estimate  $\|\partial_y z\|_{L^2_{xy}}$ and
$\|z\|_{L^2_{xy}}$. This completes the proof of the proposition.
\end{proof}

\begin{proof}[Proof of Theorem \ref{globalm}]
Let us consider the IVP
\begin{equation}\label{mIVP3}
     \left\{
\begin{array}{lll}
{\displaystyle u_t+\partial_x \Delta u+u^2u_x  =  0,  } \\
{\displaystyle  u(x,y,0)=u_0(x,y)}.
\end{array}
\right.
\end{equation}
Assume  that $u_0\in H^s(\mathbb{R}^2)$, $3/4<s<1$ (a priori) and
satisfies $\|u_0\|_{L^2}<\sqrt3\|\varphi\|_{L^2}$. We split the
datum $u_0$ as
$$
u_0(x)=(\chi_{\{|\xi|<N\}}\widehat{u}_0)^\vee(x)+(\chi_{\{|\xi|\geq
N\}}\widehat{u}_0)^\vee(x)=v_0(x)+w_0(x),
$$
where $N\gg1$ will be chosen later. First, we note that
\begin{equation}\label{c17}
\|v_0\|_{L^2}<\sqrt3\|\varphi\|_{L^2}, \qquad \|v_0\|_{H^1}\sim
N^{1-s},
\end{equation}
and
$$
\|w_0\|_{H^\rho}\sim N^{\rho-s}, \qquad 3/4<\rho \leq s<1.
$$
In view of Propositions \ref{proposition3} and \ref{proposition4},
we can solve the IVPs \eqref{mIVP1} and \eqref{mIVP2}, with initial
data $v_0$ and $w_0$, respectively, obtaining  solutions $v(t)\in
H^1(\mathbb{R}^2)$ and $w(t)\in H^\rho(\mathbb{R}^2)$, for
$t\in[0,T]$, where $T\sim N^{-2(1-s)/\gamma}$, $\gamma\in(0,5/12)$.
Moreover, the solution $u$ of \eqref{mIVP3} can be rewritten as
$$
u(t)=v(t)+U(t)w_0+z(t), \qquad t\in[0,T],
$$
where $z(t)$ is given by \eqref{z}.

Given any $\widetilde{T}>0$, our goal now is to extend the solution
$u$ on the whole interval $[0,\widetilde{T}]$ by an iteration
process.

At the point $t=T$, we have
\begin{equation}\label{c19}
u(T)=v(T)+U(T)w_0+z(T).
\end{equation}
Since $U$ is an unitary group, the function $U(T)w_0$ remains in
$H^s(\mathbb{R}^2)$. We shall show that $v(T)+z(T)$ still satisfies
the condition in \eqref{c17}.

Note that  \eqref{c19} and \eqref{mass} imply
\begin{equation}\label{c20}
\begin{split}
\|v(T)+z(T)\|_{L^2}&\leq \|u(T)-U(T)w_0\|_{L^2}\\
&\leq\|u_0\|_{L^2}+\|w_0\|_{L^2}\\
&\leq\|u_0\|_{L^2}+N^{-s}.
\end{split}
\end{equation}
Thus, for $N$ large enough, we get
$$
\|v(T)+z(T)\|_{L^2}<\sqrt3\|\varphi\|_{L^2}.
$$

Now,  Proposition \ref{proposition3} leads to
$$
\|v\|_{L^\infty_TH^1}\leq cN^{1-s},
$$
and  Proposition \ref{proposition5} tells us that
$$
\|z\|_{L^\infty_TH^1}\leq cN^{\frac{3-5s}{2}}.
$$
Hence, at each step, there is a contribution of $N^{\frac{3-5s}{2}}$
from $\|z\|_{L^\infty_TH^1}$. To reach the time $\widetilde{T}$, we
need to iterate $\widetilde{T}/T$ times, and to guarantee that the
$H^1$-norm grows on the interval $[0,\widetilde{T}]$ as $N^{1-s}$,
that is,
$$
\frac{\widetilde{T}}{T}N^{\frac{3-5s}{2}}\sim
\widetilde{T}N^{2(1-s)/\gamma}N^{\frac{3-5s}{2}}\leq c N^{1-s},
$$
we need to choose $N=N(\widetilde{T})\sim N^{\frac{10}{63s-53}}$,
for $53/63<s<1$.

Finally, from \eqref{c20}, we have a  contribution from the $L^2$-norm
of $N^{-s}$. But, since  $53/63<s<1$, for that chosen
$N(\widetilde{T})$, we deduce
$$
\widetilde{T}N^{2(1-s)/\gamma}N^{-s}\leq c
N^{\frac{63s-53}{10}}N^{2(1-s)/\gamma}N^{-s}\leq c.
$$
This proves the theorem.
\end{proof}

\subsection*{Acknowledgement}  F. L. was partially supported by
CNPq-Brazil and  A. P. was supported by CNPq/Brazil under grant
152234/2007-1.


\begin{thebibliography}{99}

\bibitem{bl} H. Berestycki and P.-L. Lions, {Nonlinear scalar field equations},
{\it Arch. Rational Mech. Anal.} {\bf 82}, (1983), 313-–375.

\bibitem{BL} H. A. Biagioni and F. Linares,  Well-posedness results for the modified Zakharov-Kuznetsov
equation, in Nonlinear Equations: Methods, Models and Applications,
Progr. Nonlinear Differential Equations Appl., 54,  Birkhäuser,
Basel, 2003, 181--189.

\bibitem{BPS} B. Birnir, G. Ponce, and N. Svanstedt,  The local ill-posedness
of the modified KdV equation, {\em Ann. Inst. H. Poincar\'e Anal.
Non Lin\'eaire},  13 (1996), 529--535.

\bibitem{BKPSV}  B. Birnir, C. E. Kenig, G. Ponce, N. Svanstedt, and L. Vega,
{\em On the ill-posedness of the IVP for the generalized Korteweg-de
Vries and nonlinear Schr\"odinger equations}, J. London Math. Soc.,
53 (1996), 551--559.

\bibitem{Bo1} J. Bourgain, Refinaments of Strichartz's inequality
and applications to 2D-NLS with critical nonlinearity,  {\it
Internat. Math. Res. Notices} {\bf 5} (1998), 253--283.

\bibitem{CKSTT} J. Colliander, M. Keel, G. Staffilani, H. Takaoka,  and T. Tao, Sharp global well-posedness for
periodic and nonperiodic KdV and mKdV, \textit{J. Amer. Math. Soc.}
{\bf16} (2003) 705--749.

\bibitem{dB} A. de Bouard,  Stability and instability of some nonlinear
dispersive solitary waves in higher dimension, \textit{Proc. Roy.
Soc. Edinburgh Sect. A} {\bf 126} (1996), 89--112.

\bibitem{Fa} A. V. Faminskii,  The Cauchy problem for the
Zakharov-Kuznetsov equation, \textit{Differ. Equ.} {\bf 31} (1995),
1002--1012.

\bibitem{FLP} G. Fonseca, F. Linares, and G. Ponce,  Global well-posedness for the
modified Korteweg-de Vries equation, {\em Comm. PDE},  {\bf24}
(1999), 683--705.

\bibitem{FLP1} G. Fonseca, F. Linares, and G. Ponce,  Global existence for
the critical generalized KdV equation, {\em Proc. Amer. Math. Soc.},
131 (2003), 1847--1855.

\bibitem{KPV} C. E. Kenig, G. Ponce, and L. Vega,  Well-posedness and scattering
results for the generalized Korteweg-de Vries equation via the
Contraction principle, \textit{Commun. Pure Appl. Math.} {\bf 46}
(1993), 527--620.

\bibitem{KZ} C. E. Kenig and S. N. Ziesler, Local well posedness
for modified Kadomstev-Petviashvili equation, \textit{Differential
Integral Equations} {\bf 18} (2005), 1111-1146.

\bibitem{KZ1} C. E. Kenig and S. N. Ziesler, Maximal function estimates with applications
to a Kadomstev-Petviashvili equation, \textit{Commun. Pure Appl.
Anal.} {\bf 4} (2005), 45--91.

\bibitem{LP} F. Linares and A. Pastor,
Well-posedness for the two-dimensional modified Zakharov--Kuznetsov
equation,  \textit{SIAM J. Math Anal.} {\bf41} (2009), 1323--1339.

\bibitem{LS} F. Linares and J.-C Saut,
The Cauchy problem for the 3D Zakharov--Kuznetsov equation,
\textit{Discrete Contin. Dyn. Syst.} {\bf24} (2009), 547--565.


\bibitem{Pa} M. Panthee, A note on the unique continuation property for
Zakharov–Kuznetsov equation,  \textit{Nonlinear Anal.} {\bf29}
(2004), 425--438.


\bibitem{We}  M. Weinstein, Nonlinear Schr\"odinger equations and sharp interpolation
estimates,  \textit{Comm. Math. Phys.}, {\bf87} (1983), 567--576.

\bibitem{ZK} V.E. Zakharov, and E.A. Kuznetsov,   On three-dimensional solitons,
{\it Sov. Phys. JETP} {\bf 39} (1974), 285--286.\\



\end{thebibliography}
\end{document}